\documentclass[11pt]{amsart}

\usepackage{amsmath,amsthm,amssymb,amssymb, paralist, xspace, graphicx, url, amscd, euscript, mathrsfs,stmaryrd,epic,eepic,color}

\usepackage[all]{xy}
\usepackage{amsmath,amscd}
\SelectTips{cm}{}

\usepackage[colorlinks=true]{hyperref}

\definecolor{brickred}{rgb}{0.8,0.25,0.33}
\definecolor{majorelleblue}{rgb}{0.38, 0.31, 0.86}

\usepackage{pstricks}

\numberwithin{equation}{section}

1

\numberwithin{subsection}{section}

\allowdisplaybreaks[1]

\newtheorem*{namedtheorem}{\theoremname}
\newcommand{\theoremname}{testing}

\theoremstyle{plain}

\newtheorem{thm}{Theorem}[section]
\newtheorem{proposition}[thm]{Proposition}
\newtheorem{proposition-definition}[thm]{Proposition-Definition}
\newtheorem{lemma-definition}[thm]{Lemma-Definition}

\newtheorem{lemma}[thm]{Lemma}
\newtheorem{conj}[thm]{Conjecture}

\theoremstyle{definition}
\newtheorem{definition}[thm]{Definition}

\newtheorem{remark}[thm]{Remark}

\newtheorem{question}[thm]{Question}
 
\newtheorem{construction}[thm]{Construction}

\theoremstyle{remark}

\numberwithin{thm}{section}

\newcommand\ocM{\overline{\mathcal{M}}}

\newcommand\cD{\mathcal{D}}

\newcommand\cL{\mathcal{L}}
\newcommand\cM{\mathcal{M}}
\newcommand\cN{\mathcal{N}}
\newcommand\cO{\mathcal{O}}

\newcommand\cX{\mathcal{X}}

\def\O{\mathcal{O}}

\def\P{\mathbb{P}}
\def\A{\mathbb{A}}

\def\u{\underline}

\newcommand\uB{\underline{B}}
\newcommand\uC{\underline{C}}
\newcommand\uD{\underline{D}}
\newcommand\uE{\underline{E}}
\newcommand\uF{\underline{F}}
\newcommand\uf{\underline{f}}

\newcommand\uS{\underline{S}}

\newcommand\uW{\underline{W}}
\newcommand\uX{\underline{X}}
\newcommand\uY{\underline{Y}}

\renewcommand\AA{\mathbb{A}}

\newcommand\NN{\mathbb{N}}
\newcommand\PP{\mathbb{P}}

\newcommand\ZZ{\mathbb{Z}}

\newcommand\fM{\mathfrak{M}}

\newcommand\arr{\ifinner\to\else\longrightarrow\fi}

\def\displaytimes_#1{\mathrel{\mathop{\times}\limits_{#1}}}

\def\displayotimes_#1{\mathrel{\mathop{\bigotimes}\limits_{#1}}}

\newcommand\Gal{\operatorname{Gal}}

\newcommand\spec{\operatorname{Spec}}

\newdir{ >}{{}*!/-5pt/@{>}}

\newcommand\doublelong[2]{\mathbin{\xymatrix{{}\ar@<3pt>[r]^{#1}
\ar@<-3pt>[r]_{#2}&}}}

\newlength{\ignora}

\newcommand{\ug}{{\underline{g}}}

\renewcommand{\setminus}{\smallsetminus}

\begin{document}


\title{$\A^1$-connected varieties of rank one over nonclosed fields}

\author{Qile Chen}

\author{Yi Zhu}

\address[Chen]{Department of Mathematics\\
Columbia University\\
Rm 628, MC 4421\\
2990 Broadway\\
New York, NY 10027\\
U.S.A.}
\email{q\_chen@math.columbia.edu}

\address[Zhu]{Department of Mathematics\\
University of Utah\\
Room 233\\
155 S 1400 E \\
Salt Lake City, UT 84112\\
U.S.A.}
\email{yzhu@math.utah.edu}

\thanks{Chen is partially supported by NSF grant DMS-1403271.}


\date{\today}
\begin{abstract}
In this paper, we proved two results regarding the arithmetics of separably $\A^1$-connected varieties of rank one. First we proved 
over a large field, there is an $\AA^1$-curve through any rational point of the boundary, if the boundary divisor is smooth and separably rationally connected. Secondly, we generalize a theorem of Hassett-Tschinkel for the Zariski density of integral points over function fields of curves. 


\end{abstract}
\maketitle

\tableofcontents

\section{Introduction}\label{sec:intro}

Separably $\A^1$-connected varieties has been introduced and studied in \cite{CZ, A1}. They are the analogue of separably rationally connected (SRC) varieties in the non-proper setting. When the non-proper variety admits a log smooth compactification, the recent developments on log stable maps provide us a powerful tool to study $\AA^1$-connectedness. We refer to \cite{KKato} for the basics of logarithmic geometry, and to \cite{GS, Chen, AC, log-bound, Wise} for the details of the theory of stable log maps.

In this paper, we study the arithmetics of simple separably $\A^1$-connected varieties of rank one with the SRC center over nonclosed fields, or equivalently, log pairs with the ambient variety smooth proper and the boundary divisor smooth irreducible SRC. Our results consist of two parts: one is over large fields and the other is over function fields of algebraic curves over an algebraically closed field of characteristic zero. 

\subsection{Over large fields}

According to Iitaka's philosophy, we expect the results for SRC varieties hold for separably $\A^1$-connected varieties in an appropriate form. Our first motivation here is to generalize Koll\'ar's theorem \cite[Theorem 1.4]{Kollar-local}: over a large field $K$, every rational point of a proper SRC variety is contained in a very free rational curve defined over $K$. In the logarithmic setting, we {would like} to find $\A^1$-curves on a proper separably $\A^1$-connected {log} variety defined over $K$. 
Since each $\A^1$-curve also gives a $K$-rational point on the boundary, a necessary condition for existence of $\A^1$-curves is $\uD(K)\neq\emptyset$. Conversely, we have the following:


\begin{thm}\label{thm:local}
Let $K$ be a large field, and $X=(\uX,\uD)$ be a proper, log smooth, simple, and separably $\A^1$-connected $K$-variety of rank one, see Section \ref{ss:notations} for the terminologies. Further assume that $\uD$ is separably rationally connected. Then there exists a very free $\A^1$-curve defined over $K$ through any $K$-rational point of $\uD$.
\end{thm}




\subsection{Over function fields}
Let $k$ be an algebraically closed field of characteristic zero. Let $B$ be a smooth projective algebraic $k$-curve, and let $F$ be its function field. Our second motivation is to study arithmetics of $\A^1$-connected varieties over $F$. Based on the work of \cite{KMM,GHS,HT06}, Hassett-Tschinkel proposed the weak approximation conjecture:
\begin{conj}\cite{HT06}
 Proper rationally connected varieties defined over $F$ satisfy the weak approximation. 
\end{conj}

Over number fields, number theorists are also interested in the approximation results for non-proper varieties, i.e. the strong approximation. Note that affine spaces satisfy the strong approximation \cite[Thm.6.13]{Rosen}. From our point of view, $\A^1$-connected varieties are generalizations of affine spaces. We propose the following question:

\begin{question}\label{strong-good}
Does strong approximation hold for $\A^1$-connected varieties over $F$?
\end{question}

A special case of Question \ref{strong-good} is the Zariski density of integral points studied by Hassett-Tschinkel \cite{HT-log-Fano}. Using the log deformation theory, we prove a stronger version of Hassett-Tschinkel's theorem in {the $\AA^1$-connectedness} setting.



 

{
\begin{thm}\label{thm:integral-pts}
Let $X=(\uX,\uD)$ be a log smooth, proper, and $\A^1$-connected variety of rank one with the SRC center defined over $F$.
Given a model $\pi: (\u\cX,\u\cD) \to B$ with the generic fiber $(\uX,\uD)$, 
let $T$ be a non-empty finite set of places on $B$ containing the images of the singularities of $\u\cX$ and $\u\cD$. 

Then for any finite set of places of good reductions $\{b_i\}_{i\in I}$ of $\pi$ away from $T$ \cite[Definition 4]{HT-log-Fano}, and  points $x_i$ in the strongly $\A^1$-uniruled locus of $\u\cX_{b_i}:=\pi^{-1}(b_i)$, there exists an $T$-integral point $\sigma:B\to \u\cX$ such that $s(b_i)=x_i$. 

In particular, $T$-integral points of the family $\pi: (\u\cX,\u\cD) \to B$ are Zariski dense.
\end{thm}
}

We define the \emph{strongly $\A^1$-uniruled locus} of a log smooth variety $X = (\uX ,\uD)$ to be the open subset of $\uX \setminus \uD$, consisting of points in the image of a free $\AA^1$-curve.



 By \cite[Corollary 1.10]{CZ}, the above theorem generalizes the previous work of Hassett and Tschinkel \cite[Theorem 1]{HT-log-Fano}. It also includes the pairs $(\P^1,{\infty})$ and Hirzebruch surface $H_n$ with the $(-n)$-curve as the boundary, where the original argument of Hassett and Tschinkel does not apply. We wish to further study Zariski density and strong approximation in our subsequent work for $\AA^1$-connected varieties with more general boundaries.

\subsection{Notations}\label{ss:notations}

In this paper, all log structures are fine and saturated \cite[Section 2]{KKato}. Capital letters such as $C, S, X, Y$ are reserved for log schemes. Their associated underlying schemes are denoted by $\uC, \uS, \uX, \uY$ respectively.

A log scheme $X$ is called of {\em rank one}, if the geometric fiber of the characteristic monoid $\ocM_{X,x} := \cM_{X,x}/\cO^*_{X,x}$ is either $\NN$ or $\{0\}$ for any geometric point $x \in X$. Given a pair $(\uX, \uD)$ with $\uD \subset \uX$ a cartier divisor, denote by $X$ the canonical log scheme associated to the pair $(\uX, \uD)$, see \cite[Complement 1]{KKato}. Such log scheme $X$ is of rank one. For simplicity, we may write $X = (\uX, \uD)$ to denote the corresponding log scheme and the underlying pair. We say that a log smooth proper variety $X=(\uX,\uD)$ of rank one is \emph{simple} if the boundary divisor is irreducible and smooth. We will keep using the terminology in \cite{AC}, and call $\uD$ \emph{the center}. 


Let $K$ be a field, and $X$ be a proper, log smooth $K$-variety defined by a log smooth pair $(\uX, \uD)$ such that $\uD \subset \uX$ is a smooth divisor. 

Given a stable log map $f: C/S \to X$ over $S$, a marking $\Sigma \subset C$ is called a {\em contact marking} the the corresponding contact order is non-trivial. See \cite[Section 3.8]{AC} and \cite{ACGM} for more detains of contact orders. 

Recall that a stable log map $f: C/S \to X$ is {\em non-degenerate} if the log structure $\cM_{S}$ is trivial over every geometric point on $\uS$. In this case, the log scheme $S$ is equipped with the trivial log structure, the underlying source curve $\uC$ is necessarily a smooth curve, and $f$ sends the locus of $C$ with the trivial log structure to the locus of $X$ with the trivial log structure. These follow from the definition of stable log maps, see for example \cite[Section 2.2]{AC} and \cite[Construction 3.3.3]{Chen}.

A stable log map is called an {\em $\AA^1$-map} if it is a non-constant, genus zero stable log map with precisely one contact marking. An $\AA^1$-map is called an {\em $\AA^1$-curve} if the corresponding stable log map is non-degenerate.


We use $\fM_{\A^1,n}(X,\beta)$ to denote the log algebraic $K$-stack of stable $\A^1$-{maps} with target $X$, $n$ non-contact markings, and curve class $\beta \in H_{2}(\uX)$. Denoted by $\u{\fM}_{\A^1,n}(X,\beta)$ its underlying algebraic stack.  When $n = 0$, we write $\u{\fM}_{\A^1}(X,\beta)$ instead of $\u{\fM}_{\A^1,0}(X,\beta)$.

Let $\u{s}$ be a $K$-point of $\uD$. Denote by $\u{\fM}_{\A^1}(X,\beta;\u{s})$ the fiber of the contact evaluation morphism
$$\u{\fM}_{\A^1}(X,\beta)\to \uD$$
over $\u{s}$.

\subsection*{Acknowledgments}
The authors would like to thank the anonymous referee for his/her detailed comments and suggestions on the manuscript.


\section{A Gluing technique}\label{sec:glue}

\begin{definition}\label{def:comb}
Let $K \subset L$ be any field extension. An {\em $\AA^1$-comb} over $\uS = \spec L$ is a stable $\A^1$-{map} $f: C/S \to X$ in $\u\fM_{\A^1}(X,\beta)(\uS)$ satisfying:
\begin{enumerate}
 \item the underlying curve $\uC$ of $C$ is given by a union of irreducible components $\uC_0, \uC_1, \cdots, \uC_m$ over $\uS$ such that $\uC$ is obtained by joining $\uC_0$ and $\uC_i$ along two $L$-points $q_i: \uS \to \uC_0$ and $p_i: \uS \to \uC_i$ for each $i \neq 0$.
 \item the unique contact marking is given by an $L$-point $q_{\infty}: \uS \to \uC_0$.
 \item the general fiber of the restriction $f_i := f|_{\uC_i}$ over $\uS$ defines a family of $\AA^1$-curves on $X$ for $i \neq 0$.
\end{enumerate}
We call $f_i$ the {\em $\AA^1$-tooth} of $f$, and $\uC_0$ the {\em handle} of $f$.
\end{definition}

We introduce an $\AA^1$-comb construction when the teeth are Galois conjugate to each other.

\begin{proposition}\label{prop:Galois-gluing}
Let $K \subset L$ be a Galois extension with $G = \Gal(L/K)$. Given an $K$-rational point $\u{s} \in \uD(K)$, and $\A^1$-curves $[f_i: C_i \to X ] \in \u\fM_{\A^1}(X,\beta;\u{s})(L)$ for $i = 1,\cdots, m$ with $m\geq 2$, such that they are contained in a $G$-orbit under the Galois action. Then there exists an $\AA^1$-comb $[f: C/S \to X] \in\u\fM_{\A^1}(X/K,\beta';\u{s})(L)$ with $S$ the standard log point over $ \spec L$ satisfying:
\begin{enumerate}
 \item $\beta' = m\cdot \beta$;
 \item $f_i$ is the tooth of $f$ for each $i$;
 \item $\uC$ is obtained by gluing $\uC_1, \cdots, \uC_m$ along $m$ different $L$-rational points of $\uC_0 = \PP^1$ contained in a $G$-orbit;
 \item $f$ contracts the handle $\uC_0$ to the $K$-rational point $\u{s}$;
 \item the log structure on $S$ is minimal in the sense of \cite{Chen, AC}.
\end{enumerate}
If furthermore the set of $\AA^1$-curves $\{[f_i]\}$ forms a complete Galois orbit, then $[f]$ is $G$-invariant, and descents to a $K$-rational point in $\u\fM_{\A^1}(X/K,\beta';\u{s})(K)$.
\end{proposition}
\begin{proof}
Comparing with the case of usual stable maps, the major difficulty is to construct morphism on the level of log structures. We split the construction into several steps.

\bigskip
{\sc Step 1. Construct the underlying map.} 

Choose $\uC_0 = \PP^1$ defined over $K$ with prescribed $m$ different $L$-rational points
\begin{equation}\label{equ:special-pts}
q_1,\cdots, q_m,
\end{equation}
and a $K$-rational point $q_{\infty}$, which will be the contact marking of the $\AA^1$-comb. We may choose the $L$-rational points in (\ref{equ:special-pts}) contained in a $G$-orbit compatible with the Galois action on $\{f_i\}$. Let $\uC$ be the nodal curve over $L$ obtained by gluing $\uC_0$ and $\uC_i$ by identifying $q_i$ with the contact marking $p_i \in \uC_i$. Then $\uf: \uC \to \uX$ is defined by gluing $\uf_i$ with the contraction map $\uf(\uC_0) = \u{s}$. 

\bigskip
{\sc Step 2. Expansion along $\uD$}

Denote by $N := N_{\uD/\uX}$ the normal bundle of $\uD$ in $\uX$, and form $\PP = \PP(N\otimes \cO_{D})$. Thus, we have a $\PP^1$-fiberation:
\[
\phi: \PP \to \uD
\]
with two disjoint sections $\uD_{0} \cong \uD_{\infty} \cong \uD$ such that 
\[
N_{\uD_0 / \PP} \cong N_{\uD_{\infty} / \PP}^{\vee} \cong N^{\vee}.
\]
Consider $\uW = \PP \cup_{\uD_0 \cong \uD} \uX$ obtain by gluing $\PP$ and $\uX$ using the canonical identification $\uD_{0} \cong \uD$. By \cite{LogSS}, there is a canonical log smooth family
\begin{equation}\label{equ:log-expansion}
\psi: W \to B
\end{equation}
over $\underline{\psi}: \uW \to \uB := \spec K$. The underlying family $\underline{\psi}$ is called the {\em expansion along $\uD$}. We call $\psi$ the {\em logarithmic expansion along $D$}. Note that we have a natural morphism of log schemes
\begin{equation}\label{equ:log-contract}
\pi: W \to X
\end{equation}
whose underlying morphism $\u{\pi}$ is the contraction of the $\PP^1$-fiberation $\phi$. This can be shown by a similar argument as in for example \cite[Proposition 6.1]{GS}.

\bigskip
{\sc Step 3. Lift $\uf$ to underlying stable map to the expansion}

Denote by $c = \beta \cap \uD \in \ZZ_{>0}$. The integer $c$ is the contact order of $f_i$ at the contact marking $p_{i}$ for each $i$. Since both $\uX$ and $\uD$ are defined over $K$, the fiber of the restriction $\phi|_{\u{s}}$ is a $\PP^1_K$ defined over $K$. We then construct the underlying stable map 
\[
\uf_0': \uC_0 \cong \PP^1 \to \PP^1_K
\]
such that 
\begin{enumerate}
 \item $\uf_{0}'$ factors through $\PP^1_{\uB}$;
 \item $\uf_{0}'$ tangent to $\uD_0$ at $q_i$ with contact order $c$ for $i = 1,\cdots,m$;
 \item $\uf_0'$ tangent to $\uD_{\infty}$ at $q_{\infty}$ with contact order $m\cdot c$.
\end{enumerate}
Such $\uf_{0}'$ can be defined by choosing a non-zero $L$-rational function in 
\begin{equation}\label{equ:rational-section}
H^0(\cO_{\PP^1}(m\cdot c \cdot q_{\infty} - \sum_i c\cdot q_i)).
\end{equation}
Gluing $\uf_0'$ and $\uf_i$ by identifying the $L$-rational points $p_i$ and $q_{i}$. , we obtain the underlying stable map $\uf': \uC \to \uW$. 

\bigskip
{\sc Step 4. Lift $\uf'$ to a stable log map to $W/B$.}

We next construct a stable log maps $f'$ over $\uf'$ as in the following commutative diagram
\begin{equation}\label{diag:lift-expansion}
\xymatrix{
C \ar[r]^{f'} \ar[d] & W \ar[d]^{\varphi} \\
S \ar[r]^{h} & B
}
\end{equation}
where $\uS \cong \uB$. 

Denote by $C^{\sharp} := (\uC, \cM^{\sharp}) \to B^{\sharp} := (\uB, \cM_{B^{\sharp}})$ the log curve with the canonical log structure over the underlying curve $\uC$. Let $\sigma_i \in \uC$ be the node obtained by gluing $q_i$ and $p_i$. By \cite{LogSS}, there is a canonical log structure $\cN_{i}$ over $\uB$ associated to the node $\sigma_i$ of the underlying curve $\uC$. Furthermore, we have 
\[
\cM^{\sharp}_{B} \cong \cN_{1}\oplus_{\cO^*}\cdots \oplus_{\cO^*}\cN_{m}.
\]

Since the log structure $\cM_B$ and $\cN_i$ are canonically associated to the underlying structure of the fibers, by the same argument as in \cite[Section 5.2.3]{Kim}, for each $i \neq \infty$ the underlying map $\uf'$ induces a morphism of log structures defined over $L$:
\begin{equation}\label{equ:node-rt}
h_i: \cM_{B} \to \cN_{i}.
\end{equation}
To construct the stable log map as in (\ref{diag:lift-expansion}), it suffices to construct a log scheme $S = (\uS, \cM_{S})$ with isomorphisms
\begin{equation}\label{equ:node-rt1}
\cM_{S} \cong \cN_i , \ \ \  \mbox{for each $i$}.
\end{equation}
Since by our construction of the morphism, the Galois action provides a canonical set of such isomorphisms by permuting the nodes and the underlying maps. This provides the log map as needed.

\bigskip

Finally, the composition $f:=\pi\circ f': C/S \to X$ is a stable log map to $X$ lifting the underlying stable map $\uf$ as in {\sc Step 1}, which fulfills the conditions as in the statement. The minimality in (5) follows from a direct calculation of the minimal monoid.

When the set of $\AA^1$-teeth forms a complete Galois orbit, we notice that the rational section of (\ref{equ:rational-section}) can be choosing defined over $K$. Since the isomorphism (\ref{equ:node-rt1}) is given by the Galois conjugation, the $\AA^1$-map is stable under the Galois action, hence descents to a $K$-rational point as in the statement.
\end{proof}



\begin{lemma}\label{lem:1}
Let $\overline{K}$ be an algebraic closure of $K$, and $X=(\uX,\uD)$ be a log smooth,  proper, simple, and separably $\A^1$-connected $\overline{K}$-variety of rank one, see Section \ref{ss:notations} for the terminologies. Further assume that $\uD$ is separably rationally connected. Then there exists a very free $\A^1$-curve over $\overline{K}$ through any $\overline{K}$-rational point of $\uD$.
\end{lemma}

\proof The proof is similar to that of \cite[Theorem 1.9]{CZ}. We give a sketch as follows. Since $\uD$ is SRC, by \cite[IV.3]{Kollar}, given any point $p\in \uD$, there exists a free rational curve $f:\P^1\to \uD$ connecting $p$ and a general point $q$. By separably $\A^1$-connectedness, we may choose a very free $\A^1$-curve $g:(\P^1,\{\infty\})\to (\uX,\uD)$ with the boundary marking $q$ such that $\deg (f^*\O_{\uX}(\uD)+g^*\O_{\uX}(\uD))>0$. By \cite[Lemma 3.6]{CZ}, we can glue $f$ and $g$ into a stable $\A^1$-map with the contact marking $p$. A general smoothing of the stable $\A^1$-map will do the job.\qed

\begin{proof}[Proof of Theorem \ref{thm:local}] 
Given a $K$-point $p\in \uD(K)$, by Lemma \ref{lem:1} we may choose a finite Galois extension $L$ over $K$ such that there exists a very free $\A^1$-curve $f_1$ passing through $p_L$. By Proposition \ref{prop:Galois-gluing}, gluing the Galois orbit of $f_1$, we obtain an $\A^1$-comb $f\in \fM_{\A^1}(X/K,m\beta;\u{s})(K)$. We may further assume that $f$ is automorphism-free. By construction, $f$ is unobstructed and the minimal log structure on $S$ has rank one. Thus, it gives a smooth point of the underlying scheme $\u\fM_{\A^1}(X/K,m\beta;\u{s})(K)$. Since the set of very free $\A^1$-curves through $p$ forms a dense open subset of $\u\fM_{\A^1}(X/K,m\beta;\u{s})$, the theorem is proved when $K$ is large.\end{proof}


\begin{remark}
When $X$ is log Fano and the normal bundle of $D$ is nontrivial and effective, there is a simple proof using Koll\'ar's result and {\cite[Lemma 3.5]{CZ}}. However, our condition is weaker. The normal bundle of $\uD$ could be negative, for example the Hirzebruch surface with the $(-n)$-curve as the boundary.
\end{remark}









\section{Zariski density}
\proof[Proof of Theorem \ref{thm:integral-pts}]
The proof contains several steps as below. 

\bigskip

{\sc Step 1. Reduction to the smooth model.} 

Given an integral model
$$\pi:(\u\cX,\u\cD)\to B$$
with the generic fiber $(\uX,\uD)$, c.f., \cite[Definition 4]{HT-log-Fano}, by the resolution of singularities of pairs, we get a proper birational morphism
$$\u\rho:\u\cX'\to \u\cX$$
such that 
\begin{itemize}
\item $\u\cX'$ is smooth and $\u\rho^{-1}(\u\cD)$ is normal crossings; 
\item denote $\u\cD'$ the proper transform of $\u\cD$, which is smooth as well;
\item $\u\rho$ is an isomorphism over the open subset of $(\u\cX,\u\cD)$ where $\u\cX$ and $\u\cD$ are both smooth. Thus $(\u\cX',\u\cD')$ is an integral model as well.
\end{itemize}
In particular, we get the following lemma.

\begin{lemma}
With the notations as above, if $T$ is a non-empty finite set of places of $B$ containing the images of the singularities of $\u\cX$ and $\u\cD$, then the set of $T$-integral points on $(\u\cX',\u\cD')$ coincides with the set of $T$-integral points on $(\u\cX,\u\cD)$.\qed
\end{lemma}

Therefore, to prove Theorem \ref{thm:integral-pts}, we may assume that both $\u\cX$ and $\u\cD$ are smooth. 

\bigskip
{\sc Step 2. Comb construction as a usual stable map.}

\begin{construction}\label{con1}
{Recall that $\{b_i\}_{i \in I}$ is a finite set of places of good reductions away from $T$.} Let $U_1$ be the divisor $\sum_{i\in I}b_i$. By assumption, for each $i\in I$, we may choose a free $\A^1$-curve 
$$g_i:=(\P^1,\infty)\to (\u\cX_{b_i},\u\cD_{b_i})$$ 
such that 
\begin{itemize}
\item $\ug_i(0)=x_i$;
\item $\ug_i(\infty)=y_i\in \u\cD$;
\item $\deg_{\P^1} \ug_i^*(\u\cD)=e_i$,
\end{itemize}  
{where $x_i$ is a point in the strongly $\A^1$-uniruled locus of $\cX_{b_i}$.}
\end{construction}
\begin{construction}\label{con2}

Pick a finite collection of general points $\{p_j\in B\}_{j\in J}$ away from $T\cup U_1$ such that for each $j\in J$, there exists a free $\A^1$-curve $$h_j:(\P^1,\{\infty\})\to  (\u\cX_{b_j},\u\cD_{b_j})$$
with $\deg_{\P^1}\u{h}_j^*(\u\cD)=e$. Let $U_2$ be the divisor $\sum_{j\in J}p_j$. 

\end{construction}

\begin{construction}\label{con3}

 Since the geometric generic fiber of $\u\cD \to \uB$ is rationally connected, by \cite{KMM,GHS}, there is a section $\uf_0: \uB \to {\u\cD}$ such that: 



\begin{itemize}
 \item $\uf_0(b_i)=y_i$ for each $i\in I$;
 \item $\uf_0(p_j)=\u{h}_j(\infty)$ for each $j\in J$;
 \item $H^1(N_{\uf_0}(-T-U_1-U_2)) = 0$, where $N_{\uf_0}$ is the normal bundle of the image $\uC_0:=\uf_0(B)$ in $\u\cD$.

\end{itemize}\end{construction}

Gluing maps $\uf_0$, $\{g_i\}_{i\in I}$ and $\{h_j\}_{j\in J}$ along the cooresponding points, we obtain the usual stable map 
$$\uf: \uC \to \u\cX.$$
We call $\uf$ a {\em comb} with {\em handle} $\uf_0$, and {\em teeth} $\{g_i\}_{i\in I} \cup \{h_j\}_{j\in J}$.

\bigskip
{\sc Step 3. Lift the comb to a stable log map.}
 

By the following lemma and \cite[Lemma 21]{HT-log-Fano}, there is a stable log map $f:C/S \to X$ lifts $\uf$ as long as the cardinality of $J$ is sufficiently large and all $p_i$'s are general.

\begin{lemma}\label{lem:lift-to-log} 
Notations as above, denote by $\uC_0 = \uf_0(\uB)$. Fix any point $\sigma\in \uC_0$ whose image in $\uB$ is away from $U_1\cup U_2$. Assume that $c \geq 0$, and there is an isomorphism 
\[N_{\u\cD/\u\cX}|_{\uC_0} \cong \cO_{\uC_{0}}(c\cdot \sigma -\sum_{i\in I}e_ib_i- e\sum_{j\in J}p_j).\] 
Then there is a log map $f:C/S \to \cX$ with a unique contact marking $\sigma$ of contact order $c$, where $\cX$ is the log scheme with the divisorial log structure associated to the pair $(\u\cX, \u\cD)$.
\end{lemma}
\begin{proof}
By assumption, we may choose a surjection 
\[
\cO_{\uC_0} \oplus N^\vee_{\u\cD/\u\cX}|_{\uC_{0}} \to \cO_{\uC_0}( \sum_{i\in I}e_ib_i+ e\sum_{j\in J}p_j)
\]
where the restriction to the first factor is given by the divisor $\sum_{i\in I}e_ib_i+ e\sum_{j\in J}p_j$, and to the second factor is given by $c\cdot \sigma$. This induces a morphism 
\[\uC_0 \to \PP(\cO_{\uC_0} \oplus N_{D/X}|_{\uC_{0}})\] 
tangent to $D_{\infty}$ at $\sigma$ of order $c$, and tangent to $D_0$ at $p_i$ of order $c_i$. By the same argument as in \cite[Lemma 3.6]{CZ}, the section $s$ induces a map to the expansion. To further lift the underlying stable map to a stable log map, we will need the set of isomorphisms of the nodes as in (\ref{equ:node-rt1}). But since we are over algebraically closed filed, we could always make a choice of such isomorphisms. This provides the stable log map as needed.
\end{proof}

\bigskip
{\sc Step 4. Deformation of the comb.}

We choose $\sigma$ in Lemma \ref{lem:lift-to-log} to be a point in $T$. Replacing $\cX$ by the product $\cX \times \P^n$, where $\P^n$ is a projective space with the trivial log structure, we may assume that the stable log map $f$ is a log immersion in characteristic zero. Here by \emph{log immersion} we mean a morphism of log schemes $f:C \to \cX$, such that the log differential $df$ is everywhere injective. Thus the log normal bundle $N_f$, the cokernel of $TC\to f^*T\cX$, is locally free.


Next we analyze the positivity of the log normal bundle $N_f(-T-\sum_{i\in I}\uf^{-1}(x_i))$. 
Define $\uE_i$, respectively, $\uF_j$, the domain component of $g_i$, respectively $h_j$. We have the following:
\begin{itemize}
\item By Construction \ref{con1}, $H^1(E_i, N_f|_{\uE_i}(-\uf^{-1}(x_i)))=0$ for each $i\in I$;
\item By Construction \ref{con2}, $H^1(F_j,N_f|_{\uF_j})=0$  for each $j\in J$;
\item By Construction \ref{con3} and \cite[(4.3.7) and Lemma 4.13]{A1}, we have that $H^1(N_{f}|_{\uC_0}(-T-U_1-U_2)) = 0$ as long as the cardinality of $J$ is sufficiently large and all $p_i$'s are general. 
\end{itemize}

By the long exact sequence of the restriction of the log normal bundle, we conclude that 
$$H^1(\uC, N_f(-T-\sum_{i\in I}\uf^{-1}(x_i)))=0.$$

Now we conclude the proof by a general deformation of $f$ using the lemma below.
\qed

\begin{lemma}
Let $X$ be a log smooth variety with boundary $D$, and $f:C\to X$ be a log immersion.  Suppose furthermore that $T'\subset \uC$ is a finite set of smooth points with trivial contact orders such that  
\begin{enumerate}
 \item $f(x) \in X \setminus D$ for any $x \in T'$, and
 \item $H^1(\uC, N_f(-T'))=0.$
\end{enumerate}
Then a general deformation of $f:C\to X$ yields a non-degenerate stable log map as in Section \ref{ss:notations}
\end{lemma}
\begin{proof}
Let $\fM$ be the stack of stable log maps containing $[f]$ with the tautological morphism induced by the minimal log structure of $\fM$: 
\[
\fM \to \cL og
\]
where $\cL og$ is the Olsson's stack parameterizing log structures over $\spec \ZZ$, see \cite{Olsson03}. By \cite[Theorem 8.31]{LogCot}, the vanishing condition 
\[H^1(\uC, N_f(-T-\sum_{i\in I}(x_i)))=0\]
implies that the tautological morphism $\fM \to \cL og$ is smooth at $[f]$. Since $\cL og$ has an open dense substack with the trivial log structure \cite[Corollary 5.25]{Olsson03}, a general deformation of $f$ yields a stable log map with the trivial log structure on its base, hence is non-degenerate.
\end{proof}








\bibliographystyle{amsalpha}             

\providecommand{\bysame}{\leavevmode\hbox to3em{\hrulefill}\thinspace}
\providecommand{\MR}{\relax\ifhmode\unskip\space\fi MR }
\providecommand{\MRhref}[2]{%
  \href{http://www.ams.org/mathscinet-getitem?mr=#1}{#2}
}
\providecommand{\href}[2]{#2}

\end{document}